\documentclass[11pt,A4]{amsart}

\usepackage{amsmath,amsfonts,amssymb}
\usepackage{verbatim}
\usepackage{amsthm}

\newtheorem{theorem}{Theorem}[section]
\newtheorem{corollary}[theorem]{Corollary}
\newtheorem{proposition}[theorem]{Proposition}
\newtheorem{lemma}[theorem]{Lemma}

\def \bN {\mathbb N}

\def \bR {\mathbb R}

\def \cF {\mathcal F}

\def \cL {\mathcal L}

\def \fh {\mathfrak h}

\def \fn {\mathfrak n}

\def \fv {\mathfrak v}

\def \fz {\mathfrak z}

\def \Ad {\text{\rm Ad}}

\def\and{\quad\mbox{and}\quad}
\def\where{\quad\mbox{where}\quad}

\begin{document}

\title[Heisenberg oscillator]
{The Heisenberg oscillator}
\author{V. Fischer}

\address{V\'eronique Fischer\\
King's College London\\ 
Strand\\
 London WC2R 2LS, UK}
\email{veronique.fischer@kcl.ac.uk}
\thanks{The author thanks the London Mathematical Society and King's College London for support.}


\begin{abstract}
In this short note, we determine the spectrum of the Heisenberg oscillator which is the operator defined as $L+|x|^2+|y|^2$ on the Heisenberg group $H_1=\bR^2_{x,y}\times \bR$ where $L$ stands for the positive sublaplacian. 
\end{abstract}

\maketitle

\section{Introduction}
The quantum harmonic oscillator on the real line:
$$
 -\partial_x^2 + x^2 
\ ,
$$
 is intimately linked with the three-dimensional real Heisenberg algebra $\fh_1$.
Indeed on the one hand the operators 
of derivation $\partial_x$ and of multiplication by $ix$ 
generate the Heisenberg Lie algebra
since their commutator 
$[\partial_x, ix ] = i$ is central;
on the other hand $-(-\partial_x^2 + x^2 )$ is the sum of the square of these two operators.

This has the following well known consequences for 
the Heisenberg group
 $H_1=\bR^2\times \bR$
 whose law is chosen here as:
 $$
 (x,y,t)(x',y',t')=(x+x',y+y',t+t'+\frac{xy'-x'y}2)
 \ .
 $$
Let $X$, $Y$ and $T$ be the three elements of $\fh_1$ forming the canonical basis of $\fh_1$;
it satisfies $[X,Y] = T$.
We identify the elements of $\fh_1$ with left invariant vector fields on $H_1$
and we define the sublaplacian: $L=-(X^2+Y^2)$.
Let $\tau$ be the representation of $H_1$ on $L^2(\bR)$ such that
$$
d\tau(X)= \partial_x
\quad,\quad
d\tau(Y)=ix
\quad\mbox{and necessarily}\quad
d\tau (T)= i 
\ .
$$
Then $\tau$ is the well known unitary irreducible Schr\"odinger representation of $H_1$ corresponding to the central character $t\mapsto e^{it}$.
Furthermore 
$$
d\tau (L)= -\partial_x^2 + x^2  \ .
$$
The spectrum of the quantum harmonic is well known
and 
this last equality allows to describe the spectrum of $L$.

In this short note,
we reverse the line of approach described above
to study the following unbounded operator  on $L^2(H_1)$:
$$
L + x^2+y^2 = -(X^2 +Y^2) +x^2 +y^2 \ ;
$$
we call this operator \emph{the Heisenberg oscillator}.
Our main result is the determination of its spectrum.

This study could very easily be generalised to the $(2n+1)$-dimensional Heisenberg group.

\medskip

In fact we will study the operator $L + \lambda_2^2(x^2+y^2)$ 
for $\lambda_2\not=0$, 
even if  by homogeneity it would suffice to study the case $\lambda_2=1$.

In  the Heisenberg oscillator 
the central variable of $H_1$ appears only 
as derivatives in the expression of the  vector fields
\begin{equation}
\label{expr_X_Y}
X=\partial_x - \frac y 2 \partial_t
\and
Y=\partial_y +\frac x 2 \partial_t
\ .
\end{equation}
This motivates our choice to study the Heisenberg oscillator 
intertwined with the Fourier transform $\cF_{\lambda_1}$ 
in the central variable of $H_1$:
\begin{equation}
\label{F_lambda1}
\cF_{\lambda_1} f(x,y)=\int_\bR e^{-i\lambda_1 t} f(x,y,t) dt
\ .
\end{equation}
Hence the object at the centre of this paper is
\begin{equation}
\label{op_intertwined_cF}
\cF_{\lambda_1} \big(L +\lambda_2^2 (x^2+y^2) \big)\cF_{\lambda_1}^{-1}
\end{equation}
where $\lambda=(\lambda_1,\lambda_2)$ with $\lambda_2\not=0$.

\medskip

The result of this note gives a complete description of the spectrum of  the operator \eqref{op_intertwined_cF}
which can also be viewed as a magnetic Schr\"odinger operator with quadratic potential. 
Some of the properties of the spectrum of that type of operators are already known by specialists of this domain (see for example \cite{hansson}) and coincide with our explicit description in the particular case of the operator \eqref{op_intertwined_cF}.
In the future the result of this note will allow the study of a Mehler type formula for the operator given by \eqref{op_intertwined_cF}, of the $L^p$-multipliers problem and of Strichartz estimates for the Heisenberg oscillator $L +(x^2+y^2)$.

\medskip

This paper is organised as follows.
First we construct a six-dimensional nilpotent Lie group $N$ 
and a representation $\rho_\lambda$ of $N$ such that 
the image of 
 the canonical sublaplacian $\cL$ of $N$ through $\rho_\lambda$
 is given by \eqref{op_intertwined_cF}.
In the third section
we study more systematically the representations of $N$ via the orbit method
and the diagonalisation of the image of $\cL$.
It allows us in the fourth section to go back to the study of the Heisenberg oscillator.
In a last section, we obtain a Mehler type formula for the operator given by \eqref{op_intertwined_cF}. 

\medskip

\textbf{Acknowledgement:}
The author is very grateful to Professors Fulvio Ricci and Ari Laptev for insightful discussions.

\section{The nilpotent Lie group\\ associated with the Heisenberg oscillator}

\subsection{The group $N$}
We consider the unbounded operators on $L^2(H_1)$
given by
the left-invariant vector fields $X$ and $Y$ (see \eqref{expr_X_Y}) and the multiplications by $ix$ and $iy$.
They generate a six-dimensional real Lie algebra 
$$ 
\fn := \bR X_1 \oplus \bR Y_1 \oplus \bR X_2 \oplus \bR Y_2 
\oplus \bR T_1 \oplus \bR T_2
\ ,
$$
whose canonical basis satisfies the commutator relations
$$
[X_1,Y_1]=T_1
\quad , \quad
[X_1,X_2]=[Y_1,Y_2]=T_2
\ ,
$$
with all the other commutators vanishing (beside the ones given by skew-symmetry).
Hence $\fn$ is a well defined two-step nilpotent Lie algebra.
It  is stratified \cite{folland_stein}
since we can decompose:
$$
\fn = \fv \oplus \fz
 \ ,
 $$
 where
the subspace
$$
\fv:=\bR X_1 \oplus \bR Y_1 \oplus \bR X_2 \oplus \bR Y_2 
\ ,
$$
generates the Lie algebra $\fn$
and the subspace
$$
\fz:=  \bR T_1 \oplus \bR T_2
\ ,
$$ 
is the centre of $\fn$.

The connected simply connected nilpotent Lie group associated with $\fn$
is $N$ identified with $\fv\times \fz\sim \bR^6$
using exponential coordinates.
 Hence $N$ is endowed with the group law
$$
(v,z)(v',z')=(v+v',z")
$$
where, for 
$v=(x_1,y_1,x_2,y_2)$,
$v=(x'_1,y'_1,x'_2,y'_2)$,
$z=(z_1,z_2)$ and $z'=(z'_1,z'_2)$, we have:
$$
z"=(z_1 +z'_1+\frac{x_1y'_1-x'_1y_1}2,
z_2+z'_2 +\frac{x_1x'_2-x_2x'_1}2 + \frac{y_1 y'_2 -y_2 y'_1}2
)
\ .
$$

We identify the elements of $\fn$ with left invariant vector fields on $N$.
We denote by 
\begin{equation}
\label{def_cL}
\cL:= -(X_1^2+Y_1^2+X_2^2+Y_2^2)
\ ,
\end{equation}
the canonical sublaplacian of $N$.

\subsection{The representation $\rho_\lambda$}
Let $\lambda=(\lambda_1,\lambda_2)$ with $\lambda_2\not=0$.
We consider the representation $d\rho_\lambda$ of the Lie algebra $\fn$
over $L^2(\bR^2)$
defined by:
\begin{equation}
\label{drho}
\left\{
\begin{array}{rcl}
d\rho_\lambda (X_1)
&=& 
\cF_{\lambda_1} X \cF_{\lambda_1}^{-1} 
=
\partial_x -i\frac y 2 \lambda_1
\\
d\rho_\lambda (Y_1)
&=& 
\cF_{\lambda_1} Y \cF_{\lambda_1}^{-1} 
=
\partial_y+ i\frac x 2 \lambda_1
\\
d\rho_\lambda (X_2) 
&=& i\lambda_2 x
\qquad
 d\rho_\lambda (Y_2) \ = \ i\lambda_2 y
 \\
d\rho_\lambda (T_1) &=& i\lambda_1 
\hspace{2.5em}
 d\rho_\lambda (T_2) \ = \ i\lambda_2 
\end{array}\right.
 \qquad .
\end{equation}

Throughout this paper, $L^2(\bR^2)$ is endowed with its natural Hilbert space structure whose Hermitian product is given by:
$$
\big(f,g\big)_{L^2(\bR^2)}
=
\int f(x,y) \bar g(x,y) dx dy
\ .
$$

It is not difficult to compute that $d\rho_\lambda$ is the infinitesimal representation of the unitary representation $\rho_\lambda$ of $N$ 
on $L^2(\bR^2)$ 
given by:
$$
\rho_\lambda(v,z)f(x,y)=
e^{i\lambda_1 (z_1+\frac { xy_1-x_1y}2)
+i\lambda_2 (z_2 + xx_2 +yy_2 +\frac{x_1x_2}2 +\frac{y_1y_2}2)}
f(x+x_1,y+y_1)
\ ,
$$
where $f\in L^2(\bR^2)$,
$(x,y)\in \bR^2$,
$(v,z)\in N$ with $v=(x_1,y_1,x_2,y_2)$ and $z=(z_1,z_2)$.

By \eqref{drho} the image of the canonical sublaplacian $\cL$ of $N$ 
(see \eqref{def_cL}) through $\rho_\lambda$ is:
\begin{equation}
\label{drho(cL)}
d\rho_\lambda(\cL) = 
\cF_{\lambda_1} \big(L +\lambda_2^2 (x^2+y^2) \big)\cF_{\lambda_1}^{-1}
\ .
\end{equation}

In the next section, we will show that $\rho_\lambda$ is equivalent to an irreducible unitary representation $\pi_\lambda$ and we will diagonalise $\pi_\lambda(\cL)$.

\section{The representations of $N$}

In this section, after describing all the unitary irreducible representations of $N$ using the orbit method \cite{corwin_greenleaf},
we obtain a diagonalisation of $\rho_\lambda(\cL)$.

\subsection{All the representations of $N$}
We need to describe the orbits of $N$ acting on the dual $\fn^*$ of $\fn$ by the dual of the adjoint action.
Each element of $\fn^*$ will be written as $\ell=(\omega,\lambda)$ 
where $\omega$ and $\lambda$ are linear forms on $\fv$ and $\fz$ respectively,
identified with a vector of $\fv$ and $\fz$ by the canonical scalar products of these two spaces.
It is not difficult to determine representatives of the co-adjoint orbits:
\begin{lemma}
\label{lem_coadjoint_representatives}
Each co-adjoint orbit of $N$ admits exactly one representative of the form
$\ell=(\omega,\lambda)$
with $\lambda=(\lambda_1,\lambda_2)$ 
satisfying
\begin{itemize}
\item[(i)]  $\lambda_2\not=0$ and $\omega=0$
\item[(ii)]  $\lambda_2=0$, $\lambda_1\not=0$, $\omega\in \bR X_2 \oplus \bR Y_2$
\item[(iii)]  $\lambda_1=\lambda_2=0$ and any $\omega$.
\end{itemize}
\end{lemma}

\begin{proof}[Sketch of the proof]
For each $z\in \fz$,
let $j_z$ be the endomorphism of $\fv$ given by:
$$
\langle j_z(v) ,v'\rangle_{\fv} = \langle z,[v,v']\rangle_{\fz}
\quad,\quad
v,v'\in \fv
\ ,
$$
where $\langle, .\rangle_{\fv}$
and $\langle, .\rangle_{\fz}$
 denote the canonical scalar products on $\fv$ and $\fz$ respectively.
In the canonical basis $\{X_1,Y_1,X_2,Y_2\}$ of $\fv$, 
the endomorphism $j_z$ is represented by:
$$
\begin{pmatrix}
0&z_1&z_2&0\\
-z_1&0&0&z_2\\
-z_2&0&0&0\\
0&-z_2&0&0
\end{pmatrix}
\quad\mbox{whose determinant is}\ z_2^4 \ .
$$
So the the range of $j_z$ is $\fv$ if $z_2\not=0$,
$\bR X_1\oplus \bR Y_1$ if  $z_2=0$ but $z_1\not=0$.

As the nilpotent Lie group $N$ is of step two, 
we compute easily for $\ell=(\omega,\lambda)$ and $n=(v_o,z_o)\in N$:
$$
\ell \circ \Ad(n^{-1}) = (\omega+ j_\lambda(v_o), \lambda)
\ ,
$$
and the previous paragraph completes the proof.
\end{proof}

It is a routine exercise to compute a representation associated with a linear form and we just give here the end result for the linear forms $\ell$ described in Lemma \ref{lem_coadjoint_representatives}.

Let $\lambda=(\lambda_1,\lambda_2)$ with $\lambda_2\not=0$ 
as in (i) of Lemma \ref{lem_coadjoint_representatives}.
The representation  $\pi_\lambda$ of $N$ over $L^2(\bR^2)$ 
given by:
\begin{eqnarray*}
\pi_\lambda(v,t) h (u_1,u_2)
&=&
e^{i\lambda_1 (t_1 + u_1y_1 +\frac{x_1y_1}2)
+i\lambda_2 (t_2 +u_1 x_2 - u_2y_1 + 
\frac {x_1x_2 }2  -\frac{y_1y_2}2)}
\\
&&\qquad
h (x_1+u_1, y_2+u_2)
\ ,
\end{eqnarray*}
 is the irreducible unitary representation associated with the linear form given by $\lambda$
(for the polarisation $\bR Y_1 \oplus \bR X_2 \oplus \bR T_1 \oplus \bR T_2$).

 Let $\lambda_2=0$, $\lambda_1\not=0$, $\omega\in \bR X_2 \oplus \bR Y_2$ 
 as in (ii) of Lemma \ref{lem_coadjoint_representatives}.
 The representation  $\pi_{\lambda_1,\omega}$ 
 of $N$ over $L^2(\bR)$ 
given by:
\begin{eqnarray*}
\pi_{\lambda_1,\omega} (v,t) h (u)
=
\exp i\lambda_1 (t_1 + u_1y_1 +\frac12 x_1y_1)
\ \exp i \langle \omega , v \rangle
\ h (x_1+u)
\ ,
\end{eqnarray*}
 is the irreducible unitary representation associated with the linear form given by $(\omega,\lambda)$.

Let $\lambda_2=\lambda_1=0$ and $\omega\in \fv$ 
 as in (iii) of Lemma \ref{lem_coadjoint_representatives}.
 The character 
$$
e^{i \langle \omega , \cdot \rangle}:(v,t)\longmapsto \exp  i \langle \omega , v \rangle
\ ,
$$
gives the one-dimensional unitary representation 
associated with the linear form given by $\omega$.

By Kirillov's methods,
the representations $\pi_\lambda$, $\pi_{\lambda_1,w} $
and $e^{i \langle \omega , \cdot \rangle}$
exhaust all the irreducible unitary representations of $N$,
up to unitary equivalence.

\subsection{The representations $\pi_\lambda$ and $\rho_\lambda$}

Let us focus on the representations $\pi_\lambda$ 
with $\lambda=(\lambda_1,\lambda_2)$, $\lambda_2\not=0$.
Its infinitesimal representation is given by:
\begin{equation}
\label{dpi}
\left\{\begin{array}{rclcrcl}
d\pi_\lambda (X_1)
&=&
\partial_{u_1}
&\qquad&
d\pi_\lambda (Y_1)
&=&
i\lambda_1 u_1 -i\lambda_2 u_2
\\
d\pi_\lambda (X_2) &=& i\lambda_2 u_1
 & \qquad&
 d\pi_\lambda (Y_2) &=&\partial_{u_2}
 \\
d\pi_\lambda (T_1) &=& i\lambda_1
 &\qquad&
 d\pi_\lambda (T_2) &=& i\lambda_2 
\end{array}\right .
\quad .
\end{equation}

We can now go back to the study of the representation $\rho_\lambda$.
Its restriction to the centre gives the character $z\mapsto e^{i\lambda(z)}$;
so by Kirillov's method, we know that $\rho_\lambda$ is equivalent to one or several copies of $\pi_\lambda$, depending whether $\rho_\lambda$ is irreducible.
In fact it is not difficult to find a concrete expression for the intertwiner between $\rho_\lambda$ and $\pi_\lambda$
(see the proposition just below)
and this shows in particular that
 $\rho_\lambda$ is irreducible.

\begin{proposition}
\label{prop_rho_eq_pi}
For each $\lambda=(\lambda_1,\lambda_2)$, $\lambda_2\not=0$,
the representations $\rho_\lambda$ and $\pi_\lambda$ are unitarily equivalent.
More precisely, let $T_\lambda=T:L^2(\bR^2)\rightarrow L^2(\bR^2)$
be the unitary operator given by:
$$
Th(x,y)
=
\sqrt{ \frac { |\lambda_2| } {2\pi}} \
e^{i \frac {\lambda_1}2 x y}
\int_\bR e^{-i\lambda_2 y z} h(x,z) dz
\ .
$$
Then 
$$
T\pi_\lambda = \rho_\lambda T
\ .
$$
\end{proposition}

\begin{proof}
The operator $T$ can be written as $T=T_1T_2$
where 
$T_1,T_2:L^2(\bR^2) \rightarrow L^2(\bR^2)$
are the unitary operators given by:
\begin{eqnarray*}
T_1f(x,y)
&=&
e^{i\frac {\lambda_1}2 xy}
f(x,y)
\ ,
\\
T_2f(x,v)
&=&
\sqrt{ \frac { |\lambda_2| } {2\pi}} 
\int_\bR e^{-i\lambda_2 v y} f(x,y) dy\ .
\end{eqnarray*}
The computations of the infinitesimal action on the canonical basis 
through
$\rho_\lambda^{(1)} = T_1^{-1}\circ \rho_\lambda \circ T_1$
and then $\rho_\lambda^{(2)} = T_2^{-1}\circ \rho^{(1)} _\lambda\circ T_2$
yield the result.
\end{proof}

\subsection{Diagonalisation of $d\pi_\lambda(\cL)$}
By \eqref{dpi}
 the image of the canonical sublaplacian through $\pi_\lambda$ is
the operator:
$$
d\pi_\lambda(\cL)
= 
-\partial_{u_1}^2 + \left(\lambda_1 u_1 -\lambda_2 u_2\right)^2
+\left(\lambda_2 u_1\right)^2 -  \partial_{u_2}^2\ ,
$$
for which we determine a diagonalisation basis.

We need to study the homogeneous polynomial of degree two:
\begin{equation}
\label{homogeneous_poly}
\left(\lambda_1 u_1 -\lambda_2 u_2\right)^2
+\left(\lambda_2 u_1\right)^2
=
u^t M_\lambda u
\ ,
\end{equation}
where
$$
u=\begin{pmatrix}u_1\\u_2\end{pmatrix}
\and
M_\lambda=
\begin{pmatrix}
\lambda_1^2 +\lambda_2^2 & -\lambda_1\lambda_2\\
-\lambda_1\lambda_2& \lambda_2^2
\end{pmatrix}
\ ,
$$
and this boils down to diagonalising the matrix $M_\lambda$.
We obtain:
$$
k^{-1}_\lambda M_\lambda k_\lambda 
=
\begin{pmatrix}
\mu_{+,\lambda} &0\\0&\mu_{-,\lambda}
\end{pmatrix}
$$
where
\begin{equation}
\label{def_mu}
\mu_{\epsilon,\lambda}
= \frac 12 \left( \lambda_1^2 +2\lambda_2^2 + \epsilon |\lambda_1|
\sqrt{\lambda_1^2 +4\lambda_2^2}\right) > 0
\quad , \quad \epsilon=\pm 
\ ,
\end{equation}
and $k_\lambda$ is the orthogonal $2\times 2$-matrix:
\begin{equation}
k_\lambda=
\begin{pmatrix}
\frac{\lambda_1\lambda_2}
{\sqrt{\left(\lambda_1\lambda_2\right)^2 
+\left(\frac{ \lambda_1^2 - |\lambda_1|\sqrt{\lambda_1^2 +4\lambda_2^2} } 2\right)^2}}
&
\frac{\lambda_1\lambda_2}
{\sqrt{\left(\lambda_1\lambda_2\right)^2 
+\left(\frac{ \lambda_1^2 + |\lambda_1|\sqrt{\lambda_1^2 +4\lambda_2^2} } 2\right)^2}}
\\
\frac{ \lambda_1^2 - |\lambda_1|\sqrt{\lambda_1^2 +4\lambda_2^2} } 
{2\sqrt{\left(\lambda_1\lambda_2\right)^2 
+\left(\frac{ \lambda_1^2 - |\lambda_1|\sqrt{\lambda_1^2 +4\lambda_2^2} } 2\right)^2}}
&
\frac{ \lambda_1^2 + |\lambda_1|\sqrt{\lambda_1^2 +4\lambda_2^2} } 
{2\sqrt{\left(\lambda_1\lambda_2\right)^2 
+\left(\frac{ \lambda_1^2 + |\lambda_1|\sqrt{\lambda_1^2 +4\lambda_2^2} } 2\right)^2}}
\end{pmatrix}
\ .
\end{equation}

The change of variable
\begin{equation}
\label{change_var_u'_u}
u'= k_\lambda u
\quad ,\quad
u'=\begin{pmatrix}u'_1\\u'_2\end{pmatrix}
\and
u= \begin{pmatrix}u_1\\u_2\end{pmatrix}
\ ,
\end{equation}
transforms the homogeneous polynomial \eqref{homogeneous_poly}
into $\mu_{+,\lambda}{u'_1}^2
+ \mu_{-,\lambda}{u'_2}^2$
and leaves the 2-dimensional laplacian invariant,
that is,
$-(\partial_{u_1}^2 + \partial_{u_2}^2)=
-(\partial_{u'_1}^2 + \partial_{u'_2}^2)$;
the operator  $\pi_\lambda(\cL)$
becomes:
\begin{equation}
\label{pi(cL)2}
\pi_\lambda(\cL)
=
-\partial_{u'_1}^2 
-  \partial_{u'_2}^2
+ \mu_{+,\lambda}{u'_1}^2
+ \mu_{-,\lambda}{u'_2}^2
\quad, \quad u'= k_\lambda u
\ .
\end{equation}

Recall that the Hermite functions $h_m$, $m\in \bN$,
defined by:
$$
h_m(x)=e^{-\frac {x^2}2} H_m(x)
\where
H_m(x)=(-1)^m e^{x^2} \frac {d^m}{dx^m}(e^{-x^2})
\ ,
$$
form an orthonormal basis of $L^2(\bR)$ which diagonalises the quantum harmonic oscillator:
$$
-h"_m(x) + x^2 h_m= (2m+1) h_m
\ .
$$

Using the notation above, we obtain:
\begin{proposition}
\label{prop_diag_pi(cL)}
The operator $\pi_\lambda(\cL)$ admits the following orthonormal basis of eigenfunctions:
$$
h_{\lambda, m} (u):= 
|\lambda_2|^{-1/2}
h_{m_+} (\mu_{+, \lambda}^{1/4} u'_1)
\ h_{m_-} (\mu_{-, \lambda}^{1/4} u'_2)
$$
where $m=(m_+,m_-)\in \bN^2$
and $u'= k_\lambda u$. 
The eigenvalue associated with $h_{\lambda,m}$ is 
$$
\nu_{\lambda, m}:=\mu_{+, \lambda}^{1/2}  (2m_+ +1)
+ \mu_{-, \lambda}^{1/2} (2m_- +1)
\ .
$$
\end{proposition}

Consequently,
by Proposition \ref{prop_rho_eq_pi}, 
we obtain:
\begin{corollary}
\label{cor_diag_rho(cL)}
The operator given by \eqref{drho(cL)}, that is,
$$
d\rho_\lambda(\cL) = 
\cF_{\lambda_1} \big(L +\lambda_2^2 (x^2+y^2) \big)\cF_{\lambda_1}^{-1}
\ ,
$$
admits 
$\{Th_{\lambda,m}, \ m\in \bN^2\}$ as orthonormal basis of eigenfunctions
and 
the eigenvalue associated with $Th_{\lambda,m}$ is 
$\nu_{\lambda, m}$.
\end{corollary}

\section{Spectrum of $L -\lambda_2^2(x^2+y^2)$}

For any $f\in L^2(H_1)$, $\lambda=(\lambda_1,\lambda_2)$, $\lambda_2\not=0$, and $m\in \bN^2$,
we define:
\begin{equation}
\label{def_c}
c_{\lambda,m} (f)
:=
\big( \cF_{\lambda_1} f, Th_{\lambda,m}\big)_{L^2(\bR^2)}
\ ,
\end{equation}
where $\cF_{\lambda_1}$ is the Fourier transform \eqref{F_lambda1} in the central variable 
and $Th_{\lambda,m}$ the orthonormal basis of $L^2(\bR^2)$ given in Corollary 
\ref{cor_diag_rho(cL)}.

\begin{lemma}
\label{lem_c(L+)}
We have for any $f\in L^2(H_1)$ such that 
$\big(L+\lambda_2^2(x^2+y^2)\big) f \in L^2(H_1)$:
$$
c_{\lambda, m}\left(\big(L+\lambda_2^2(x^2+y^2)\big) f\right) =
\nu_{l,m} c_{\lambda,m}(f)
\ .
$$
\end{lemma}
\begin{proof}
Recall
$$
\cF_{\lambda_1} 
\left(\big(L+\lambda_2^2(x^2+y^2)\big) f\right) 
=
d\rho_\lambda(\cL) \cF_{\lambda_1} f
\ .
$$
As $d\rho(\cL)$ is self-adjoint,
we have:
\begin{eqnarray*}
&&c_{\lambda, m}\left(\big(L+\lambda_2^2(x^2+y^2)\big) f\right) 
=
\big( d\rho_\lambda(\cL)\cF_{\lambda_1} f, Th_{\lambda,m}\big)_{L^2(\bR^2)}
\\
&&\qquad
=
\big( \cF_{\lambda_1} f, d\rho_\lambda(\cL)Th_{\lambda,m}\big)_{L^2(\bR^2)}
=
\bar \nu_{l,m}\big( \cF_{\lambda_1} f, Th_{\lambda,m}\big)_{L^2(\bR^2)}
\\
&&\qquad
=
\nu_{l,m} c_{\lambda,m}(f)
\ ,
\end{eqnarray*}
by Corollary \ref{cor_diag_rho(cL)}.

\end{proof}

Now we fix $\lambda_2\in \bR\backslash \{0\}$.
For any Borelian set $B$ of $\bR$, 
let $E(B)$ be the operator defined on $L^2(H_1)$ by
$$
E(B) f
=
\cF_{\lambda_1}^{-1}
\left[\sum_{m\in \bN^2} 
1_{\nu_{\lambda,m} \in B} c_{\lambda,m}(f)
Th_{\lambda,m}\right]
\ ,
$$
where $c_{\lambda,m} (f)$ is defined by \eqref{def_c}.
With Lemma \ref{lem_c(L+)},
it is a routine exercise to check that
$B\mapsto E(B)$ is the spectral resolution of $L+\lambda_2^2(x^2+y^2)$.
The spectrum is:
$$
\left\{ \nu_{(\lambda_1,\lambda_2),m}, \ \lambda_1\in \bR, \ m\in \bN^2
\right\}
=
[\nu_{(0,\lambda_2), 0},+\infty)
\ ,
$$
where
$$
\nu_{(0,\lambda_2), 0}
=
\mu_{+,(0,\lambda_2)}^{1/2}
+\mu_{-,(0,\lambda_2)}^{1/2}
=
2|\lambda_2|
\ .
$$

\section{Application: Mehler type formulae}

The Mehler formula \cite[Theorem.12.63]{simon}
states that the integral kernel of 
the operator $\exp \left(-t \big(-\partial_x^2 + x^2 -1 \big)\right)$
is:
 $$
Q_t(x,y) = \pi^{-\frac12} (1-e^{-4t})^{-\frac 12} 
\exp \left( -F_t(x,y)
 \right)
\ ,
$$
where
$$
F_t(x,y)
=
 (1-e^{-4t})^{-1} 
\left( \frac 12 (1+e^{-4t})(x^2+y^2) -2e^{-2t}xy\right)
\ .
$$

Hence for any $\mu>0$,
 the integral kernel of $\exp \left(-t \big( -\partial_x^2 +\mu x^2 \big)\right)$ 
is:
$$
K_{t,\mu}(x,y)=\sqrt{\mu} e^{-t\mu} Q_{t\mu} (\sqrt \mu x, \sqrt \mu y)
\ .
$$

We conclude this note with the following Mehler type formulae for the operators $d\pi_\lambda (\cL)$ and 
$d\rho_\lambda (\cL)=\cF_{\lambda_1} \big(L +\lambda_2^2 (x^2+y^2) \big)\cF_{\lambda_1}^{-1}$ (given by
\eqref{drho(cL)}):

\begin{proposition}
The integral kernel of the operator 
$\exp \big(-td\pi_\lambda (\cL)\big)$ 
is:
$$
\kappa_{t,\lambda}\big( (u_1,u_2),(v_1,v_2)\big)
=
K_{t,\mu_{+,\lambda}}(u_1,v_1)K_{t,\mu_{-,\lambda}}(u_2,v_2)
  .
$$

The integral kernel of  the operator $\exp \big(-td\rho_\lambda (\cL)\big) $ is:
\begin{eqnarray*}
&&Q_{t,\lambda} \big((x_o,y_o), (x,y)\big)\\
&&\quad =
\frac{|\lambda_2|}{2\pi}
e^{i\frac{\lambda_1}2 (x_o y_o - x y) }
\int_{\bR^2}
 e^{i\lambda_2( y_2y -y_oy_1)} 
\kappa_{t,\lambda} \big((x_o,y_1), (x,y_2)\big)
dy_1 dy_2
\end{eqnarray*}

\end{proposition}

\begin{proof}
The first formula is easily obtained from \eqref{pi(cL)2}.

For the second formula, we see that, 
by Proposition \ref{prop_rho_eq_pi},
we have:
$$
\exp \big(-td\rho_\lambda (\cL)\big) 
=
T \exp \big(-td\pi_\lambda (\cL)\big) T^{-1}
\ ,
$$
 the operators $T$ and $T^{-1}$ having integral kernels:
\begin{eqnarray*}
C_T\big((x,y), (x',y')\big) 
&=&
\sqrt{ \frac{|\lambda_2|}{2\pi} }
e^{i\frac{\lambda_1}2 xy } e^{-i\lambda_2 yy'} \delta_{x'=x}
\ ,
\\ 
C_{T^{-1}}\big((x,y), (x',y')\big) 
&=&
\sqrt{ \frac{|\lambda_2|}{2\pi} }
e^{-i\frac{\lambda_1}2 xy' } e^{i\lambda_2 yy'} \delta_{x'=x}
\ .
\end{eqnarray*}
So the operator $\exp \big(-td\rho_\lambda (\cL)\big) $ 
has integral kernel:
\begin{eqnarray*}
&&Q_{t,\lambda} \big((x_o,y_o), (x,y)\big)\\
&&\quad =
\int 
C_T\big( (x_o,y_o),(x_1,y_1) \big) 
\kappa_{t,\lambda} \big((x_1,y_1), (x_2,y_2)\big)
C_{T^{-1}}\big( (x_2,y_2),(x,y) \big) 
\\
&& \hspace{24em}dx_1 dy_1 dx_2dy_2
\\
&&\quad =
\frac{|\lambda_2|}{2\pi}
\int_{\bR^2}
e^{i\frac{\lambda_1}2 x_o y_o } e^{-i\lambda_2 y_oy_1} 
\kappa_{t,\lambda} \big((x_o,y_1), (x,y_2)\big)
e^{-i\frac{\lambda_1}2 x y } e^{i\lambda_2 y_2y } 
dy_1 dy_2
\end{eqnarray*}

\end{proof}

\bibliographystyle{amsplain}

\end{document}